\documentclass[12pt]{article}
\usepackage{version}
\usepackage{amsmath}
\usepackage{amsfonts}
\usepackage{amssymb}
\usepackage{amsthm}
\usepackage{txfonts}
\usepackage[all]{xy}
\usepackage{eucal}

\def\OO{{\mathcal O}}
\def\Q{\mathbb Q}

\def\P{\mathbb P}

\def\hn{\mathop{\rm HN}\nolimits}

\let\ov\overline

\newtheorem{theorem}{Theorem}
\newtheorem{proposition}[theorem]{Proposition}
\newtheorem{lemma}[theorem]{Lemma}
\newtheorem{definition}[theorem]{Definition}

\theoremstyle{remark}
\newtheorem{remark}[theorem]{Remark}

\begin{document}

\baselineskip=16pt

\centerline{\Large\bf Restriction Theorems for Principal Bundles and Some Consequences}
\bigskip

\centerline{\bf Sudarshan Gurjar}

\bigskip

\begin{abstract}
The aim of this paper is to give a proof of the restriction theorems for principal bundles with a reductive algebraic group as structure group in arbitrary characteristic.  
Let $G$ be a reductive algebraic group over any field $k=\bar{k}$, let $X$ be a smooth projective 
variety over $k$, let $H$ be a very ample line bundle on $X$ and let $E$ be a semistable (resp. stable) principal $G$-bundle on $X$ w.r.t. $H$. 
The main result of this paper is that the restriction of $E$ to a general smooth curve which is a complete intersection of ample hypersurfaces of sufficiently high degree's 
is again semistable (resp. stable).

\end{abstract}

2010 Mathematics Subject Classification: 14D06, 14F05 \\ 

key words: Principal $G$-bundle, Harder-Narasimhan type, canonical reduction

\section{Introduction}
Around 1981-82, V. Mehta and A. Ramanathan proved the following important theorem (see [3]) : Let $X$ be a smooth projective variety over $k=\bar{k}$ 
with a chosen polarisation and let $V$ be a torsion free coherent sheaf on it. Then the restriction of $E$ to a general, non-singular curve which is a complete-intersection 
of ample hypersurfaces of sufficiently high degree's is again semistable. The corresponding theorem for principal bundles with a reductive algebraic group as
structure group over a field of characteristic zero follows immediately since the semistability of a principal $G$-bundle in characteristic
zero is equivalent to the semistability of its adjoint bundle and hence the semistable restriction theorem for the adjoint bundle implies 
the theorem for the principal $G$-bundle as well. V. Mehta and A. Ramanathan later also gave a proof of the stable restriction theorem 
for torsion-free coherent sheaves (see [4]). However the semistable restriction theorem in positive characteristic and the stable restriction theorem in any characteristic 
remain open for principal bundles. The aim of this paper is to prove these two restriction theorems. The basic idea of the proofs is similar to that of the
semistable restriction theorem in [3]. The proofs given here are characteristic free. The paper is arranged as follows:\\

\vskip\medskipamount 
\leaders\vrule width \textwidth\vskip0.4pt 
\vskip\medskipamount 
\nointerlineskip

\bigskip

\noindent Sudarshan Gurjar\\
e-mail id: sgurjar@ictp.it\\
International Centre for Theoretical Physics (ICTP), Trieste\\

\indent In section 2, we introduce some preliminary notions and set up some notations which will be used throughout the paper.\\
\indent In section 3 we recall a degeneration argument which is central to the proofs in [3] and [4] and also draw some consequences out of it.\\
\indent In section 4 we prove the semistable restriction theorem for principal bundles using the degeneration argument introduced in section 3.\\
\indent In section 5 we prove the stable restriction theorem for principal bundles analogues to the stable restriction theorem for torsion-free
coherent sheaves proved in [4]. The proof given here is different and substantially simpler than the proof of the stable restriction theorem in [4].

\section{Preliminaries}



In this section we set up some notation and recall some basic facts which will be used in the paper. Many of these have been 
taken from [1] with only minor changes.
$X$ will always stands for a smooth projective variety defined over a field $k=\bar{k}$ of arbitrary
characteristic. $H$ will denote the chosen polarisation on $X$. Let  $G\supset B \supset T$ be a 
reductive group, together with a chosen Borel subgroup and a maximal 
torus. 
As usual, $X^*(T)$ and $X_*(T)$ will respectively denote the groups of all 
characters and all $1$-parameter subgroups of $T$. We choose once for all, a Weyl group invariant positive definite 
bilinear form on $\Q\otimes X^*(T)$ taking values in $\Q$. This, 
in particular, will allow us to identify $\Q\otimes X^*(T)$ with 
$\Q\otimes X_*(T)$.
Let $\Delta \subset X^*(T)$ be the corresponding simple roots.
Let $\omega_{\alpha} \in \Q\otimes X^*(T)$ denote the
fundamental dominant weight corresponding to $\alpha \in \Delta$,
so that $\langle \omega_{\alpha}, \beta^{\vee}\rangle = \delta_{\alpha, \beta}$
where $\beta^{\vee} \in \Q\otimes X_*(T)$ is the simple coroot corresponding 
to $\beta \in \Delta$. Note that each $\omega_{\alpha}$ is a non-negative
rational linear combination of the simple roots $\alpha$.
Recall that {\it the closed positive Weyl chamber
$\ov{C}$} is the subset of $\Q\otimes X_*(T)$
defined by the condition that
$\mu \in \ov{C}$ if and only if  
$\langle \alpha, \mu\rangle \ge 0$ for all $\alpha \in \Delta$.  
The standard {\it partial order} on $\ov{C}$  
is defined by putting $\mu \le \pi$ 
if $\omega_{\alpha}(\mu) \le \omega_{\alpha}(\pi)$ for all $\alpha \in \Delta$,
and $\chi(\mu) = \chi(\pi)$ for every character $\chi$ of $G$.

By definition, all roots and weights in $X^*(T)$ are trivial on
the connected component $Z_0(G)\subset T$ of the center of $G$.



\bigskip

\centerline{\bf Canonical reductions of principal bundles}

\medskip

Let $T\subset B\subset G$ be as above. Recall that 
a principal $G$-bundle $E$ on $X$ is said to be {\it semistable} (resp. {\it stable}) w.r.t. $H$  if for any 
reduction $\sigma : U\to E/P$ of the structure group to a parabolic $P\subset G$ defined on a large open set $U \subseteq X$ 
(one whose complement in $X$ has codimension atleast 2)
and any dominant character $\chi: P\to k^*$, the line bundle $\chi_*\sigma^*E$ on $U$ has non-positive degree (resp. negative degree).\\
\indent Note that although $L_{\sigma}$ is only defined on $U$, it extends uniquely (upto a unique isomorphism)
to a line bundle on all of $X$ and hence its degree is well-defined.\\
\indent Note that to test the semistability (resp. stability) of a principal $G$-bundle, it suffices to consider reductions to 
standard maximal parabolics (i.e those containing the chosen Borel $B$).\\

\begin{definition}
A {\it canonical reduction} of a principal $G$-bundle $E$  
is a pair $(P,\sigma)$ where $P$ is a standard parabolic subgroup
of $G$  and 
$\sigma : U\to E/P$ is a reduction of the structure group to $P$ 
such that the following two conditions hold. 

{\bf 1} If $\rho: P \to L = P/U$ is the Levi quotient of $P$ (where $U$ is
the unipotent radical of $P$) then the principal $L$-bundle 
$\rho_*\sigma^*E$ is semistable.

{\bf 2} For any non-trivial character $\chi: P \to k^*$ 
whose restriction to the chosen maximal torus $T\subset B \subset P$ 
is a non-negative linear combination $\sum n_i\alpha_i$ 
of simple roots $\alpha_i \in \Delta$ with at least
one $n_i \ne 0$, we have $\deg(\chi_*\sigma^*E) > 0$.

\end{definition}

It has been shown by Behrend (see [5]) (see [6] for a different, more bundle-theoretic proof when char $k=0)$ that when $X$ is a non-singular curve, 
each principal $G$-bundle on $X$ admits a unique canonical reduction.


Given a reduction $(P,\sigma)$ of $E$, we get 
an element $\mu_{(P,\sigma)}\in \Q\otimes X_*(T)$ defined by 
\begin{eqnarray*}
\langle \chi, \mu_{(P,\sigma)}\rangle & = & \left\{
\begin{array}{ll}
\deg(\chi_*\sigma^*E) & \mbox{if }  \chi \in X^*(P),  \\
0                     & \mbox{if } \chi \in I_P.
\end{array}\right.
\end{eqnarray*}

\noindent If $(P,\sigma)$ is the canonical reduction of $E$, then the element
$\mu_{(P,\sigma)}$ is called the {\it HN type of $E$}, and is denoted by 
$\hn(E)$. If $\alpha \in \Delta -I_P$, then 
$\langle \alpha,\text{HN}(E)\rangle =  \text{deg}(\alpha_*\sigma^*E)\geqq0$.
As $\langle \beta,\hn(E)\rangle = 0$ for all $\beta \in I_P$, we see that 
$\hn(E)$ is in the closed positive Weyl chamber $\ov{C}$, in fact, 
in the facet of $\ov{C}$ defined by the vanishing of all $\beta \in I_P$.

Note that a principal bundle $E$ of type $\hn(E)= \mu$ is semistable
if and only if $\mu$ is central, that is, $\mu = a\nu$ for
some $1$-parameter subgroup $\nu : k^* \to Z_0(G)$ and $a\in \Q$.  

Given the HN-type $\mu = \hn(E)$ of $E$, we can recover the corresponding
standard parabolic $P$ as follows. Let $I_{\mu} \subset \Delta$ 
be the set of all simple roots $\beta$ such that $\langle \beta ,\mu\rangle =0$.
Then $I_{\mu}$ is exactly the set of inverted simple roots which defines $P$. 
Alternatively, let $n \ge 1$ be any integer such that $\nu = n \mu \in X_*(T)$.
Then the $k$-valued points of $P$ are all those $g$ for which 
$\lim_{t\to 0} \nu(t)g\nu(t)^{-1}$ exists in $G$.

Let $E$ be a principal $G$-bundle on $X$, let
$(P,\sigma)$ be its canonical reduction, and let $(Q,\tau)$ be any
reduction to a standard parabolic. Let $\hn(E) = \mu_{(P,\sigma)}$
and $\mu_{(Q,\tau)}$ be the corresponding elements of $\Q\otimes X_*(T)$
(the element $\hn(E)$ lies in the closed positive Weyl chamber $\ov{C}$, 
but $\mu_{(Q,\tau)}$ need not do so). Then for each $\alpha \in \Delta$ 
we have the inequality
$$\langle \omega_{\alpha}, \hn(E)\rangle = 
\langle \omega_{\alpha}, \mu_{(P,\sigma)}\rangle \ge  
\langle \omega_{\alpha}, \mu_{(Q,\tau)}\rangle $$
where $\omega_{\alpha} \in \Q\otimes X^*(T)$ is the fundamental dominant weight
corresponding to $\alpha$.
Moreover, if each of the above inequalities is an equality then 
$(P,\sigma) = (Q,\tau)$. \\

\noindent We now recall a basic fact regarding semistability of bundles in families:\\

\begin{lemma} (See [2], Proposition 5.10))\label{Openness of Semistability}
 Let $S$ be a finite-type irreducible scheme over $k$ and let $f: Z \rightarrow S$ be a smooth, projective morphism of relative dimension one. Let $E$ be a 
 principal $G$-bundle on $Z$. Let $\eta$ denote the generic point of $S$ and $Z_{\eta}$ denote the generic fibre of $f$. 
 If the $HN(E_{\eta})=\mu$, then there exists a non-empty open subset $U \subseteq S$ such that $\forall s \in U, HN(E_s)= \mu$.
 In particular, if $E\mid_{Z_\eta}$ is semistable, then there exists a non-empty open set $U \subseteq X$ such that $\forall s \in U$, $E\mid_{Z_s}$ is semistable.
\end{lemma}

\bigskip
We now recall some facts regarding complete intersection subvarieties from [3].\\
\\
\indent Let $X$ and $H$ be as before. For any non-negative integer $m$, let $S_{\bf m}$ denote the projective space 
$\P(H^\circ(X,H^m))$. For any $r$-tuple of integers $(m_1,\cdots,m_r)$, where $0 < r <n$, let $S_{\bf m}$ denote the product $S_{{\bf m}_1}\times \cdots S_{{\bf m}_r}$.
Let $Z_{\bf m} \subset X \times S_{\bf m}$ denote the correspondence variety defined by $Z_{\bf m} = \{(x,s_1,\cdots,s_r) \in X \times S_{\bf m} \mid 
s_i(x)=0~\forall~ x \in X\}$. Thus we have the following diagram:\\

 \[
\xymatrix{
        X \times S_{\bf m} \ar@{}[r]|-*[@]{\supset} &  Z_{\bf m} \ar[r]^{q_m} {\ar[d]_{p_m}}  &  S_{\bf m} \\  
        & X &   }
        \]
\bigskip

\noindent where $p_m$ and $q_m$ are the projections. The fiber over a point $(s_1,\cdots,s_r) \in S_{\bf m}$ can be thought of as a closed subscheme
of $X$ embedded via $p_m$ and defined by the ideal generated by $(s_1,\cdots,s_r)$ in the homogeneous coordinate ring of $X$. Such a subscheme
will be called a subscheme of type $m$.
$p_m$ is a fibration with the fiber over a point $x \in X$ embedded in $S_{\bf m}$ via $q_m$ as the product of hyperplanes $H_1 \times \cdots \times H_r$ where $H_i \in S_{{\bf m}_i}$ 
consists of those sections vanishing at $x$. This shows that $Z_{\bf m}$ is non-singular.\\
\indent By Bertini'theorem, the generic fiber of $q_m$, say $Y_{\bf m} \hookrightarrow S_{\bf m}$ thought of as a closed subscheme of $X$ via $p_m$ is a
geometrically irreducible, smooth, complete intersection subscheme of $X$. $Y_{\bf m}$ will be called the generic subscheme of type ${\bf m}$.
In the case when $r=n-1$, we call $Y_{\bf m}$, the generic complete-intersection curve of type $\bf m$.\\

\indent We now recall two important propositions from [3].

\begin{proposition}\label{picard group}
 Let dim $X=n\geqq2$.  Let ${\bf m}= (m_1,\cdots, m_r)$ with $1 \leqq r \leqq n-1$ be a $r$-tuple of integers with each $m_i \geqq 3$. Then the natural map Pic$(X) \rightarrow$ Pic$(Y_{\bf m})$ is a bijection.
 \end{proposition}
 
 \begin{proposition}({\bf Enrique-Severi})\label{Enrique-Severi lemma}
 Let $X \subseteq \mathbb{P}^n$ be a non-singular projective variety corresponding to $H$. Let $E$ be a vector bundle on $X$. 
 Then there exists an integer $m_\circ$ such that if ${\bf m} = (m_1, \cdots, m_r),1\leqq r<n$ is an $r$-tuple of integers with each $m_i\geqq m_\circ$, then for a non-empty 
 open subset $U_{\bf m} \subseteq S_{\bf m}$, if $s=(s_1,\cdots,s_r) \in U_{\bf m}$ is any point,
 then the restriction map $H^\circ(X,E) \rightarrow H^\circ(X_s, E\mid_{X_s})$, where $X_s$ is the closed subscheme of $X$ defined by the ideal $(s_1,\cdots,s_r)$, is a bijection.
 \end{proposition}


 \begin{remark}({\bf Enrique-Severi for a bounded family})\label{Enrique-Severi for bounded families}
 It follows easily from proposition \ref{Enrique-Severi lemma} that if $E_i$ is a bounded family of vector bundles on $X$, then again
 one can find an integer $m_\circ$ satisfying the property given in proposition \ref{Enrique-Severi lemma}.
 \end{remark}
 
 \bigskip  
 \begin{proposition} \label{upper-semicontinuity}
 Let $A$ be a discrete valuation ring with quotient field $K$ and residue field $k$. 
 Let $S=$ Spec $A$. Let $Y \rightarrow S$ be a flat, projective morphism having relative dimension 1 such that $Y$ and the
 generic curve are both non-singular and the special curve is reduced with non-singular irreducible components $C_1, \cdots, C_r$. Let $E$ be a principal
 $G$-bundle on $Y$. Let $Y_K$ and $Y_k$  denote the generic and the special curve respectively.  
 Let $E_K$ and $E_k$ denote the restrictions of $E$ to the generic and special special curve respectively. Let $E^i_k$ denote the restriction of $E_k$ 
 to the component $C_i$.
 Let suppose the standard parabolics underlying the canonical reduction of $E_K$ and the canonical reductions of $E^i_k$ 
 are the same for each $i$, say $P$. Let $\chi$ be a fundamental weight corresponding to a non-inverted simple root of $P$.  Let $E_{P_K}$ (resp. $E^{i}_{P_k}$) 
 denotes the principal $P$-bundle underlying the canonical reduction of $E_K$ on $Y_K$ (resp. $E^i_k$ on $C_i$). Then we have
 deg ($\chi_*E_{P_K}) \leq \underset{i=1}{\overset{r}\sum} \text{deg } (\chi_*E^i_{P_k})$.
 \end{proposition}
 
 \begin{proof}
Since $\chi^{-1}$ is an anti-dominant character, we see that
the associated line bundle $L_{\chi^{-1}}$ on $G/P$ is very ample and hence generated by global sections. Hence, the evaluation map
$e:H^\circ(G/P,L_{\chi^{-1}}) \rightarrow L_{\chi^{-1}}\mid_{[eP]}$ is 
surjective, where $L_{\chi^{-1}}\mid_{[eP]}$ denoted the fiber of $L_{\chi^{-1}}$ over the identity coset $[eP]\in G/P$. Hence the dual map
$L_{\chi^{-1}}\mid_{[eP]}^{\vee} \rightarrow H^\circ(G/P,L_{\chi^{-1}})^{\vee}$ is injective and hence defines a one-dimensional subspace of 
$H^\circ(G/P,L_{\chi^{-1}})^{\vee}$. The scheme-theoretic stabilizer of this subspace is exactly $P$ and on it, $P$ acts by the character $\chi$.
Thus we have found a representation $\rho: G \rightarrow GL(V)$ with the property that 
$V$ has a one-dimensional subspace whose scheme-theoretic stabilizer in $G$ is exactly 
$P$ and on which $P$ acts by the character $\chi$. Let $Q$ be the maximal parabolic in $GL(V)$ stabilizing this one-dimensional subspace.
We thus get a closed embedding $i: G/P \hookrightarrow GL(V)/Q$. \\
 \indent Let $F$ denote the principal $GL(V)$-bundle obtained by extension structure group of $E$ via $\rho$ and let 
 $F(V)$ denote the rank $n$ vector bundle 
 corresponding to $F$. Let $F_K$ (resp. $F_k$) denote the restriction of $F$ to $Y_K$ (resp. $Y_k$). Similarly let
 $F(V)_K$ (resp. $F(V)_k$) denote the restrictions of $F(V)$ to $Y_K$ (resp. $Y_k$). Corresponding to the embedding $i$, we get a closed
 embedding of $E/P \hookrightarrow F/Q$. Since $Q$ is a stabilizer of a 1-dimensional subspace of $V$, it follows that $F/Q$ is naturally the same as $\P(F(V))$. 
 Since any reduction of structure group of $E$ to $P$ corresponds functorially to a section of $E/P$, any $P$-reduction of $E$ naturally induces a 
 $Q$-reduction of $F$, or equivalently, gives a line sub-bundle of $F(V)$. \\
 \indent Let $\alpha_K: Y_K \rightarrow E_K/P$ denote the canonical reduction of $E_K$ over $Y_K$.
  By the valuative criterion of properness applied to the projection $E/P$ over $S$, we see that this reduction spreads to a reduction 
  $\alpha: Y\setminus F \rightarrow E/P$
  where $F$ is a finite subset of $Y_k$. Let $\alpha_k$ denote the restriction of $\alpha$ to $Y_k \setminus F$.
  Again by applying the valuative criterion to the projection $E_k/P \rightarrow Y_k$, we see that the reduction $\alpha_k$ extends to a reduction,
  call it $\gamma_k$, over all of $Y_k$. Let $E^{i'}_{P_k}$ denote the restrictions of the principal $P$-bundle on $Y_k$
  corresponding to $\gamma_k$ to the curves $C_i$.
  Let $\tilde{\alpha_K}$
 (resp. $\tilde{\gamma_k}$) denote the corresponding $Q$-reduction on $F_K$ 
 (resp. $F_k$) obtained by composing the reduction $\alpha_K$ (resp. $\gamma_k$) via the closed embedding $E_K/P \hookrightarrow F_K/Q$
 (resp. $E_k/P \hookrightarrow F_k/Q$). Let $L_{\chi_K}$ (resp. $L_{\chi_k}$) denote the line sub-bundles of $F(V)_K$ (resp. $F(V)_k$)
 corresponding to $\tilde{\alpha}_K$ (resp. $\tilde{\gamma}_k$). It is easy to see that $L_{\chi_K}$ is isomorphic to $\chi_*E_{P_K}$.
 By the properness of the Quot scheme, the 
 short exact sequence\\ $ 0 \rightarrow L_{\chi_K} \rightarrow F(V)_K \rightarrow F(V)_K/L_{\chi_K} \rightarrow 0$ over $Y_K$ can be completed to an exact sequence
 $0 \rightarrow L_{\chi} \rightarrow F(V) \rightarrow F(V)/L_{\chi} \rightarrow 0$ over $Y$ such that $F(V)/L_{\chi}$ is flat over $S$. Let $L_{\chi_k}$
 denote the restriction of $L_{\chi}$ to $X_k$ which can be thought of as a rank 1 subsheaf of $F(V)_k$. Let $\tilde{L}_{\chi_k}$
 denote the saturation of $L_{\chi_k}$, i.e. the line sub-bundle of $F(V)_k$  obtained by pulling up the torsion on $X_k$ in the cokernel $F(V)_k/L_{\chi_k}$.
 Then it is easy to see that $\tilde{L}_{\chi_k}$ is the line bundle obtained by extension of structure group of $\gamma_k$ 
 via $\chi$. Clearly deg $L_{\chi_K} =$ deg $L_{\chi_k} \leq$ deg $\tilde{L}_{\chi_k}= \underset{i=1}{\overset{r}\sum}\text{deg } (\chi_*E^{i'}_{P_k}) \leqq \underset{i=1}{\overset{r}\sum} \text{deg } (\chi_*E^i_{P_k})$,
 where the last inequality is by the definition of the canonical reduction. 
 Hence deg $(\chi_*E_{P_K}) \leq \underset{i=1}{\overset{r}\sum} \text{deg }( \chi_*E^i_{P_k})$, thereby completing the proof of the lemma. \qed
                                                                                   
 \end{proof}
\bigskip

 \begin{remark}: The above proof also shows that if $E_k$ is semistable restricted to every irreducible component of $Y_k$, 
 then $E_K$ is also semistable. It also follows from the proof of the above proposition that if the inequality in the above proposition is actually an 
 equality, then the canonical reduction $E_{P_K}$ on the generic curve spreads to give a $P$-reduction on all of $Y$ and whose restriction to every
 irreducible component of the special curve coincides with the canonical reduction on that component.
 \end{remark}
 
\section{A degeneration argument}
 
 In this section we recall a basic result from [3] regarding  degenerating family of curves and draw some easy consequences from it.
 \bigskip

 \bigskip

\begin{proposition}(See [3], Proposition 5.2) \label{degeneration}
 
 Let $l= m+r$ . Let $U_{\bf m} \subseteq S_{\bf m}$ and $U_{\bf l} \subseteq S_{\bf l}$ be non-empty open subsets. Then there exists a point $s \in S_{\bf l}$
 and a smooth curve $C$ contained in $S_{\bf l}$ passing through $s$ such that:\\
 i) $C-{s} \subset U_{\bf l}$\\
 ii) $q_{l}^{-1}(C)$ is non-singular and $q_{l}^{-1}(C) \rightarrow C$ is flat.\\
 iii) The fibre $q_{l}^{-1}(s)$ is a reduced curve with $\alpha^r$ non-singular components $C_1, \cdots ,C_{\alpha^{r}}$ intersecting 
transversally and such that atmost two of them pass through any point of $X$ and such that each $C_i$ is a fibre of $q_m$ over a point 
of $U_{\bf m}$.
 
\end{proposition}

\begin{lemma}\label{Semistability for higher types}
 If $E\mid_{Y_{\bf m}}$ is semistable, then $E\mid_{Y_{\bf l}}$ is semistable for any $l \geq m$.
 
\end{lemma}

\begin{proof}
 By the openness of semistability (see lemma \ref{Openness of Semistability}), we see that since $E\mid_{Y_{\bf m}}$
 is semistable, there exists a non-empty open set $U_{\bf m}$ of $S_{\bf m}$ such that $E\mid_{q_m^{-1}(s)}$ is semistable
 for any $s\in U_{\bf m}$. By proposition \ref{degeneration}, we can degenerate $Y_{\bf l}$ into a reduced curve  such that
 every irreducible component of the special curve belongs to $U_{\bf m}$. As in the proof of proposition \ref{upper-semicontinuity}, we see that 
any $P$-reduction on $E\mid_{Y_{\bf l}}$ contradicting the semistability of $E\mid_{Y_{\bf l}}$ induces a $P$-reduction on the special curve 
whose restriction to each of the irreducible components contradicts the semistability of $E$ restricted to those components and hence by the choice of $U_{\bf m}$,
contradicts the semistability of $E\mid_{Y_{\bf m}}$. \qed
\end{proof}

\begin{lemma}\label{HN-type of restriction}
Let $X, H$ be as before. Let $E$ be a principal $G$-bundle on $X$. Let $C$ be a smooth curve on $X$. Then $E$ is semistable (resp. stable) if its restriction to $C$ is semistable (resp. stable).
 \end{lemma}

\begin{proof}
 Suppose $E$ is not semistable (resp. stable). Let $\sigma:U \rightarrow E/P$ be a parabolic reduction of $E$ contradicting the semistability (resp. stability) of $E$, where $U$
 is a large open subset of $X$. Let $C'$ be any smooth curve contained in $U$.
 Then $\sigma$ restricts to give a $P$-reduction on $C'$ contradicting the semistability (resp. stability) of $E\mid_C$. By constructing a 
 degenerating family of curves as in proposition \ref{upper-semicontinuity} with the generic curve contained in $U$ and with the special curve $C$, we see as in proposition \ref{upper-semicontinuity}
 that this reduction induces a $P$-reduction on $C$ as well which contradicts the semistability (resp. stability) of $E\mid_C$.   \qed
\end{proof}


\begin{lemma}\label{boundedness of subschemes}
Let $X,H$ be as before. Let $(\alpha_1,\cdots,\alpha_t)$ be any $t$-tuple of integers with $1\leqq t \leqq n-1$ and each $\alpha_i\geq3$.
Let $\{X_i\}$ be a bounded family of closed subschemes of $X$. Then there exists an integer $m_\circ$ such that 
$\forall m\geqq m_\circ$, the generic complete-intersection subvariety $Y_m$ of type $(\alpha_1^m,\alpha_2^m,\cdots,\alpha_t^m)$ (having codimension $t$ in $X$) is not contained in $X_i$ for any $i$.
\end{lemma}

\begin{proof}
The proof is by induction on $t$. Without loss of generality one may assume that $X_i$'s are all hypersurfaces. Since they form a bounded family, their degree's
are bounded above by some integer $d$. Hence the lemma clearly holds for $t=1$. Let us assume that the result holds for $t-1$.
Let $Y_m=H_1\cap H_2\cap \cdots \cap H_t$ be the complete-intersection of generic hypersurfaces of degree's
$\alpha_1^m,\cdots,\alpha_t^m$ respectively. Suppose $Y_t\subset X_{i_\circ}$, for some $X_{i_\circ}\in \{X_i\}$. By the induction hypothesis,
$H_1\cap H_2\cap \cdots \cap H_{t-1}$ is not contained in $X_{i_\circ}$ and hence by irreducibility of $H_1\cap H_2\cap \cdots
 \cap H_{t-1}$, $H_1\cap H_2\cap \cdots \cap H_{t-1}\cap X_{i_\circ}$ is a closed subscheme of $H_1\cap H_2\cap \cdots \cap H_{t-1}$ 
of codimension atleast $t$ in $X$. Since it contains $Y_m$ as a closed subscheme, its codimension in $X$ is exactly $t$. Thus it follows that
$Y_m$ is an irreducible component of $H_1\cap H_2\cap \cdots \cap H_{t-1}\cap X_{i_\circ}$. But simply by comparing degree's we see that when $m$
exceeds $d$, this is not possible. This completes the proof of the lemma $\forall t$. \qed

\end{proof}
\bigskip

\noindent {\bf Notation}: 
For the remainder of the paper, we fix an ($n$-1)-tuple of integers $(\alpha_1,\cdots,\alpha_{n-1})$ with each $\alpha_i \geq 3$. Let $\alpha = \alpha_1 \cdots \alpha_{n-1}$. For
a positive integers $m$, we denote by ${\bf m}$ the sequence $(\alpha_1^m,\cdots,\alpha_{n-1}^m)$. Henceforth we will denote by $Y_{\bf m}$,
the complete-intersection curve of type $(\alpha_1^m,\cdots,\alpha_{n-1}^m)$. As remarked earlier, this is a geometrically irreducible, non-singular curve
of degree $\alpha^m$.

\section {Semistable Restriction Theorem}

We now state and prove the semistable restriction theorem for principal bundles.

 \begin{theorem}\label{Semistable Restriction Theorem}
  Let $X$ and $H$ be as before. Let $E$ be a principal $G$-bundle on $X$. Then there exists an integer $m_\circ$ such that 
  $\forall m \geq m_\circ$, $E\mid_{Y_{\bf m}}$  is again semistable.
 \end{theorem}

 \begin{proof}
 The proof of the theorem is by contradiction. For any non-negative integer $m$, let $E_m$ denote the restriction of $E$ to $Y_{\bf m}$.
 As remarked earlier in lemma \ref{Semistability for higher types}, if $E_m$ is semistable for some $m$, then it is semistable $\forall l\geqq m$.
 So assume that the restriction $E_m$ is non-semistable for all $m$. Since there are only finitely many standard parabolics, we can find a
 sequence of integers $\{m_k\}$ such that the standard parabolic underlying the canonical reduction of 
 $E\mid_{Y_{m_k}}$ is the same $\forall k$, say $P$. By proposition \ref{Openness of Semistability}, for each $m\in \{m_k\}$,
 we can find a non-empty open subset $U_{\bf m}\subseteq S_{\bf m}$ such that 
 the $HN\mid_{q_m^{-1}(s)}$ is constant for all $s \in U_{\bf m}$. For any $m\in \{m_k\}$, let $E_{m_P}$ denote the principal $P$-bundle underlying the canonical reduction 
 of $E_m$. Fix a fundamental weight $\chi$ of $P$. Let $\chi_*E_{m_P}$ denote the line bundle obtained by extension of structure group
 of $E_{m_P}$ by $\chi$. Let $L_m$ denote the extension of $\chi_*E_{m_P}$ to all of $X$ (see proposition \ref{picard group}). Let $d_m= \text{deg } L_m$.
 Note that deg $L_m = \alpha^m\cdot\text{deg }\chi_*E_{m_P}$. \\
\indent Claim : The sequence of integers $d_{m_k}$ is a decreasing function of $k$.\\
Proof of claim :  Let $l = m + r$ with $m,l \in \{m_k\}$ and $r>0$. As before,
degenerate $Y_l$ into a reduced curve with $\alpha^r$ non-singular, irreducible components $C_1, \cdots, C_{\alpha^r}$, with $C_i = q_m^{-1}(s_i)$ 
for some $s_i \in U_{\bf m}$, intersecting transversally such that atmost two components intersect at any
given point. Let $E^i_P$ denote the canonical reduction of $E\mid_{C_i}$.
By proposition \ref{upper-semicontinuity}, we see that deg $(\chi_*E_{l_P}) \leq \underset{i=1}{\overset{\alpha^r}\sum} \text{deg } (\chi_*E^i_P) = \alpha^r\text{deg } (\chi_*E_{m_P})$, where the last equality is because of the choice of $U_{\bf m}$.
This immediately implies that $d_l \leq d_m$. Thus we see that $d_{m_k}$ is decreasing function of $k$ thereby completing the proof of the claim.\\
\indent Since by the definition of canonical reduction $d_{m_k}\geqq 0$, we conclude that the set of line bundles $L_{m_k}$ form a bounded family.\\  
\indent Consider the closed embedding $i: E/P \hookrightarrow F/Q$ described before in lemma \ref{upper-semicontinuity}, where $F$ is a principal $GL(V)$-bundle and $Q$ is a maximal parabolic in $GL(V)$
stabilizing a one-dimensional subspace on which $P$ acts by the character $\chi$. Thus we have for each $m$, $F/Q=\P(F(V))=\P(L_m^*\otimes F(V))$ .
Since the line bundles $L_{m_k}$ form a bounded family, by remark \ref{Enrique-Severi for bounded families}, 
there exists an integer $m_1$ such that for a curve of type ${\bf m}$ in $\{m_k\}$ with $m \geqq m_1$, the restriction map Hom $(L_m, F(V)) \rightarrow $ Hom $(L_m\mid_{Y_{\bf m}}\rightarrow F(V)\mid_{Y_{\bf m}})$
is a bijection. For each $m\in\{m_k\}$, consider the schematic inverse images of the cone over $E/P$ inside $L_m^*\otimes F(V)$
defined by the embedding $i$ via all the global sections of $L_{m_k}^*\otimes F(V)$. Since $L_{m_k}$'s form a bounded family, we see that this collection of closed subschemes of $X$ for all sections and for all possible $k$
form a bounded family and hence by lemma \ref{boundedness of subschemes}, it follows that there exists an integer $m_\circ \geqq m_1$ such that this collection cannot contain any
curve of type $\geqq m_\circ$ unless that subscheme containing this curve equals all of $X$. \\ 
\indent Let $m \geqq m_\circ$ be any integer in $\{m_k\}$ and consider the 
the section $\varphi_m: Y_{\bf m} \rightarrow E/P$ corresponding to the canonical reduction of $E_m$. Via the embedding $i$, this may be viewed as
a section $Y_{\bf m} \rightarrow F/Q$ and hence corresponds to a line sub-bundle of $F(V)\mid_{Y_{\bf m}}$ which can be easily seen to be isomorphic to $L_m\mid_{Y_{\bf m}}$ and hence we get a section, say
$\varphi_m$ of Hom $(L_m\mid_{Y_{\bf m}}, F(V)\mid_{Y_{\bf m}})$. Lift this 
to a section $\tilde{\varphi_m}$ of Hom $(L_m, F(V))$. Let $V_m$ be the open set where this lifted section is non-zero. On $V_m$, we get a line sub-bundle of
$F(V)$ and hence a section $\tilde{\varphi_m}: V_m \rightarrow F/Q$ which extends the section $\varphi_m$ on $Y_{\bf m}$. Note that since $V_m$ contains a curve which is a complete-intersection of ample hypersurfaces,
namely $Y_{\bf m}$, it follows that $V_m$ is a large open set. Let $X_m$ denote the scheme-theoretic inverse image of the cone over $E/P$ inside $L_m^* \otimes F(V)$ via $\tilde{\varphi_m}$. It is easy to see that $X_m$ is a closed
subscheme of $X$ containing $Y_{\bf m}$ and hence by the choice of $m$, equals all of $X$. This shows that on $V_m$, the section  $\tilde{\varphi_m}$ actually factors through a section
$V_m \rightarrow E/P$ (via the embedding $i$) lifting the canonical reduction on $Y_{\bf m}$. This section has the property that the line bundle obtained by extension of structure group
via $\chi$ has positive degree thereby contradicting the semistability of $E$. This contradiction shows that $E\mid_{Y_{\bf m}}$ is semistable
thereby completing the proof of the theorem. \qed

\end{proof}

\bigskip 

In fact in the above proof it can be shown that the line bundles $L_m$ with $m\gg 0$ are all isomorphic. Although, we do not need this fact for the proof of the above theorem
or for the rest of the paper, it is an interesting fact in itself.\\
\indent Choose, as we may by the proof of lemma above, an integer $s$ so that $\forall m\geqq s$, $d_m$ is constant. Choose any $l,m \geqq s$.
As in the above proof, consider a degenerating family $D\rightarrow S$, where $S$ is a dvr and the generic fiber is $Y_l$ and the special
fiber has $\alpha^{l-m}$ irreducible, non-singular components in $U_{\bf m}$ (in the notation of the proof). By the remark following lemma \ref{upper-semicontinuity}, it follows that the canonical 
reduction $E_{P_l}$ spreads to a $P$-reduction $\tilde{E_{P_l}}$ on all of $D$ and whose restriction to every irreducible component of the special fiber 
coincides with the canonical reduction there. Thus $L_l\mid_D$ and $\chi_*\tilde{E_{P_l}}$ are two line bundles on $D$ which are 
isomorphic restricted to the generic fiber and have the same degree restricted to every irreducible component of the special fiber.
Hence they are isomorphic on all of $D$. This implies that $L_l$ is isomorphic to $L_m$ on all the components of $D_k$ and hence by proposition
\ref{picard group}, they are isomorphic on $X$. 

\begin{proposition}({\bf Openness of Semistability in higher relative dimensions}) \label{Openness of Semistability in higher relative dimensions}

 Let $\pi:Z \rightarrow S$ be a smooth, projective morphism, where $S$ is a finite-type $k$-scheme. Let $\OO_{Z/S}(1)$ be a relatively very ample line bundle on $Z$. Let $E$ be a principal $G$-bundle on $Z$.
Let $Z_{\eta}$ denote the generic fiber of $\pi$. Suppose $E\mid_{Z_{\eta}}$ is semistable w.r.t. $\OO_{Z_\eta}(1)$.
Then there exists a non-empty open subset $U\subseteq S$ such that $\forall s\in U$, $E\mid_{Z_s}$ is semistable w.r.t $\OO_{Z_s}(1)$.
\end{proposition}

\begin{proof}
Choose a closed subscheme $C \hookrightarrow Z$ such that the restricted morphism $\pi': C\rightarrow S$ is a smooth, projective morphism of relative dimension 1
and such that the generic curve $C_{\eta}\hookrightarrow Z_{\eta}$ is a complete-intersection of general ample hypersurfaces in $Z_\eta$
of sufficiently high enough degree's so that by theorem \ref{Semistable Restriction Theorem}, the restriction of $E$ to $C_{\eta}$ is again semistable.
Hence by lemma \ref{Openness of Semistability}, there exists a non-empty open subset $U \subseteq S$ such that $\forall s \in U$, 
$E\mid_{\pi'^{-1}(s)}$ is again semistable. Then by lemma \ref{HN-type of restriction}, it follows that $E\mid_{Z_s}$ is also semistable $\forall s \in U$.\qed
\end{proof}

 \section{Stable Restriction Theorem}
 
 The aim of this section is to prove the stable restriction theorem for principal bundles. 
 The stable restriction theorem for torsion-free sheaves was proved in [4] and is as follows:
 
  \begin{theorem}([4])
  With notation as in theorem \ref{Semistable Restriction Theorem}, let $E$ be a stable torsion-free coherent sheaf on $X$ w.r.t. $H$. Then there exists
  an integer $m_\circ$ such that for any $m \geqq m_\circ$, the restriction of $E$ to $Y_{\bf m}$ is again stable.
  \end{theorem}
  
 Once again, as in the case of the semistable restriction theorem, by openness of stability, it follows that there exists a non-empty open subset $U_{\bf m} \subseteq S_{\bf m}$ such that for any $s\in U_{\bf m}$, $E\mid_{q_m^{-1}(s)}$ is again stable.\\

  We begin by proving the openness of stability for a principal $G$-bundle defined over a smooth family of curves.

 \begin{lemma}({\bf Openness of Stability for a family of curves}) \label{Openness of Stability}
 Let $S$ be a scheme of finite type over $k$ and let $f: Z \rightarrow S$ be a smooth, projective morphism of relative dimension one. 
 Let $E$ be a principal $G$-bundle on $Z$. Let $\eta$ denote the generic point of $S$ and $Z_{\eta}$ denote the generic fibre of $f$. Suppose 
 $E\mid_{Z_\eta}$ is stable. Then there exists a non-empty open subset $U\subseteq S$ such that $\forall s \in U$, $E\mid_{Z_s}$ is stable.
\end{lemma}
 
 \begin{proof}
 Since the family of curves is flat, the genus is 
 constant in the family, say $g$.
 By openness of semistability (see lemma \ref{Openness of Semistability}), we know that there exists a neighbourhood $V$ of $\eta$ such that 
 for any $s\in V$, the restriction $E_s$ is again semistable. Hence the open subset of $U$ parametrizing stable bundles is the set of
 points $s\in V$ for which $E_s$ admits a reduction to some maximal parabolic $P$ such that the line bundle obtained by 
 extension of structure group via the unique fundamental weight of $P$ has degree zero. By Riemann-Roch, the Hilbert polynomial of $Z_s$
 for any $s\in S$ w.r.t. such a line bundle is $n+1-g$. \\
 \indent Fix a maximal parabolic $P\subseteq G$. Let $\mathcal L$ denote the line bundle on $E/P$ corresponding to the fundamental weight of $P$.
 Consider the Hilbert scheme $\text{Hilb}^{n+1-g,\mathcal L}_{E/P/Z/S}$ which represents the functor from $S$-schemes to sets, associating to any $S$-scheme $T$,
 the set of all closed subschemes of $E_T/P$ which are flat over $T$ and whose restriction to every schematic-fibre has Hilbert polynomial $n+1-g$.   
 There exists an open subset ${\text{Hilb}^\circ}^{n+1-g, \mathcal L}_{E/P/Z/S}\subseteq \text{Hilb}^{n+1-g,\mathcal L}_{E/P/Z/S}$ which 
 parametrizes those subschemes of $E/P$ which are sections with this property. By properness of $\text{Hilb}^{n+1-g, \mathcal L}_{E/P/Z/S}$ over $S$,
 its image is a closed subspace of $S$. Since $E\mid_{Z_s}$ is stable, it follows that the image of ${\text{Hilb}^\circ}^{n+1-g, \mathcal L}_{E/P/Z/S}$ is a locally closed subset
 of $S$ which misses the generic point.
 The union of all these locally closed subschemes for all standard maximal parabolics is a locally closed subset in $S$
 whose complement in $V$ contains an open set $U$ parametrizing stable bundles. \qed
\end{proof} 
 
\begin{lemma}
With notation as above, if there exists some $m$ such that $E_{m}$ is stable, then $\forall l\geqq m$, $E_l$ is also stable.
\end{lemma}

\begin{proof}
By lemma \ref{Openness of Stability}, let $U_{\bf m}$ denote the non-empty open subset of $S_{\bf m}$ consisting of points $s\in S_{\bf m}$ such that the $E\mid_{q_{m}^{-1}(s)}$
is again stable. Let $l= m +r$. Degenerate $Y_l$ into a reduced curve $C$ with $\alpha^r$ many components, $C_1,\cdots,C_{\alpha^r}$, intersecting transversally, each of which is in $U_{\bf m}$ and such that 
atmost two of them intersect at any point. Suppose $E_l$ is not stable. Since $E_l$ is semistable (see remark following lemma \ref{upper-semicontinuity}), there exists a $P$-reduction $E_{l_P}$ of $E_l$ and a dominant character $\chi$ of $P$ such that 
line bundle obtained by extension of structure group has degree zero. As in the proof of lemma \ref{upper-semicontinuity}, we see this $P$-reduction on $Y_l$ induces
a $P$-reduction on $C$. Let $E^i_{P}$ denote its restriction to $C_i$. Then as in the proof of lemma \ref{upper-semicontinuity}, we see that
deg $(\chi_*E_{l_P})\leqq \underset{i=1}{\overset{\alpha^r}\sum} \text{deg } (\chi_*E^i_{P})$. Since by the choice of $U_{m}$, $E\mid_{C_i}$
is stable we immediately see that the right hand right of this inequality is strictly less than zero and hence so is the left hand side. This contradiction shows that $E_l$ is stable as well. \qed
\end{proof}

 \begin{theorem}{(Stable Restriction Theorem)}\label{Stable Restriction Theorem}
 Let $X, H$ and $G$ be as before and let
 $E$ be a principal $G$-bundle on $X$ which is stable w.r.t. $H$. Then there exists an integer $m_\circ$ such that $\forall~m\geqq m_\circ$, the restriction
 $E\mid_{Y_{\bf m}}$ is again stable.\\
 \indent Consequently by lemma \ref{Openness of Stability}, there exists a non-empty open subset $U_m \subseteq S_{\bf m}$ such that for any 
 $s\in U_m$, the restriction of $E$ to $q_m^{-1}(s)$ is again stable.
 \end{theorem}
 
 \begin{proof}
 
 By theorem \ref{Semistable Restriction Theorem}, there exists an integer $m_1$ such that $\forall m \geqq m_1$, $E\mid_{Y_{\bf m}}$ is again semistable.
The proof now is by contradiction. By lemma \ref{Openness of Stability}, we see that if the restriction $E_m$ is stable for some $m$
then it is stable for all $l\geqq m$. So assume that $E_m$ is not stable for any $m$.
Since there  are only finitely many standard parabolics, choose a sequence of increasing integers $\{m_k\}$ with each $m_k\geqq m_1$
such that there exists a standard maximal parabolic, say $P$, with the property that $\forall k$, there exists a $P$-reduction
$\psi_k: Y_{m_k} \rightarrow E/P$ contradicting the stability of $E_{m_k}$. Let $\chi$ denote the fundamental weight corresponding to $P$. Since
$E_{m_k}$ is semistable $\forall k$, we see that the line bundle $\chi_*\psi^*_k (E)$
has degree $0$. As in the proof of lemma \ref{upper-semicontinuity}, by using the Chevalley semi-invariant
lemma, we get a $G$-equivariant embedding of fiber-bundles $i:E/P \hookrightarrow F/Q$, where $F$ is a
principal $GL(V)$-bundle and $Q$ is a maximal parabolic in $GL(V)$ stabilizing a 1-dimensional subspace of $V$ on which $P$ acts by the character $\chi$. 
Via the embedding $i$, we can think of $\psi_k$ as sections of $F/Q$ and hence we get line sub-bundles of $F(V)\mid_{Y_{m_k}}$.
By lemma \ref{picard group}, extend these to line bundles on all of $X$. 
Since these line bundles all have degree zero, they form a bounded family and hence by choosing these $m_k$'s to be large enough,
we can assume that the restriction map Hom$ (L_{m_k}, F(V)) \rightarrow \text{Hom} (L_{m_k}\mid_{Y_{m_k}}, F(V)\mid_{Y_{m_k}})$
is a bijection. The rest of the proof is similar to that of theorem
\ref{Semistable Restriction Theorem}. We just sketch it briefly for the sake of completeness: By taking schematic inverse
images of the cone over $E/P$ in $L_{m_k}^*\otimes F(V)$ defined by the embedding $i$, via all the global sections of $L_{m_k}^*\otimes F(V)$ for all possible $k$,
we get a bounded family of closed subschemes of $X$ and hence by lemma \ref{boundedness of subschemes} it cannot contain curves of arbitrarily large types. Thus eventually 
for some $m_\circ \gg 0$, the reduction $\psi_{m_\circ}$ thought of as a section of Hom $(L_{m_\circ}\mid_{Y_{m_\circ}}, F(V)\mid_{Y_{m_\circ}})$ via $i$ has a lift to a section 
of Hom$(L_{m_\circ}, F(V))$ which on the large open subset $V \subseteq X$  where it is non-zero (containing $Y_{m_\circ}$) actually comes from a section $V \rightarrow E/P$ lifting the reduction 
$\psi_{m_\circ}$ on $Y_{m_\circ}$. This section naturally has the property that the line bundle obtained by extension of structure group via the unique fundamental
weight corresponding to $P$ has degree zero thereby contradicting the stability of $E$. This contradiction shows that $E\mid_{Y_{\bf {m_\circ}}}$ is stable thereby completing the proof of the theorem.  \qed
\end{proof} 
 
 \bigskip
 
\begin{remark}({\bf Openness of stability in higher relative dimensions}) The stable restriction theorem immediately implies the openness of stability for a principal $G$-bundle over a 
smooth family of projective schemes over a finite-type 
$k$-scheme. The proof is the same as the proof of lemma \ref{Openness of Semistability in higher relative dimensions} with theorem \ref{Semistable Restriction Theorem}
and lemma \ref{Openness of Semistability} in the proof replaced by theorem \ref{Stable Restriction Theorem} and lemma \ref{Openness of Stability} respectively.
\end{remark} 
 
 \bigskip
 


\noindent {\bf Acknowledgements}: I am very grateful to Yogish Holla, both for his initial encouragement to work on this problem as well as for 
going through my proof and suggesting changes, including pointing out a serious error in an earlier version.


\bigskip

\begin{thebibliography}{1111}

\bibitem[1]{1}
Gurjar, Sudarshan ; Nitsure, Nitin. Schematic HN-stratification for families of principal bundles and lambda modules. arXiv: 1208.5572.

\bibitem[2]{2}
Holla, Yogish. Parabolic reductions of principal bundles. arXiv: 0204219.

\bibitem[3]{3}
Mehta, V.B ; Ramanathan, A. Semistable sheaves on projective varieties and their restrictions to curves. Math Ann. 258 (1981/82), no. 3, 213-224.

\bibitem[4]{4}
Mehta, V.B ; Ramanathan, A. Restrictions of stable sheaves and representations of the fundamental group. Invent Math. 77 (1984), no. 1, 163-172.

\bibitem[5]{5}
Behrend, Kai A. Semi-stability of reductive group schemes over curves. Math. Ann. 301 (1995), no. 2, 281-305. 

\bibitem[6]{6}
Biswas, Indranil; Holla, Yogish I. Harder-Narasimhan reduction of a principal bundle. Nagoya Math. Journal. 174 (2004), 201-223. 

\end{thebibliography}
\end{document}